\documentclass[letterpaper, 10 pt, conference]{ieeeconf}  % Comment this line out if you need a4paper

\IEEEoverridecommandlockouts                              % This command is only needed if 
\usepackage[utf8]{inputenc}
\usepackage{amssymb}
\usepackage{amsmath}
\usepackage{bm}
\usepackage{hyperref}
\usepackage{tikz}
\usepackage{tkz-euclide}
\usepackage{chngcntr}
\usepackage{verbatim}
\usepackage{mathtools}
\usepackage{esvect}
\usepackage{subfiles}
\usepackage{color}
\usepackage{xspace}
\usepackage{xparse}
\usepackage[noend]{algpseudocode}
\usepackage{subfigure}
\usepackage{bbm}
\usepackage{enumerate}

\usepackage{pgfplotstable}
\usepackage{pgfplots}
\pgfplotsset{
	table/search path={plot_figures},
}
\usepackage{tikz}
\usepackage{xr}
\usetikzlibrary{external}
\usetikzlibrary{arrows,%
	petri,%
	topaths}%
\usetikzlibrary{graphs,graphs.standard}
\usepackage{tkz-berge}

\usepackage[percent]{overpic}

\newtheorem{theorem}{Theorem}
\newtheorem{lemma}[theorem]{Lemma}

\newtheorem{corollary}[theorem]{Corollary}

\newtheorem{assumption}{Assumption}
\newtheorem{definition}{Definition}

\DeclareMathOperator*{\argmin}{arg\,min}

\usepackage{pgfplotstable}
\usepackage{pgfplots}
\pgfplotsset{
	table/search path={plot_figures},
}
\usepackage{tikz}
\usepackage{xr}
\usetikzlibrary{arrows,%
	petri,%
	topaths}%
\usetikzlibrary{graphs,graphs.standard}
\usepackage{tkz-berge}

\title{\LARGE \bf
On Increasing Self-Confidence in Non-Bayesian Social Learning \\ over Time-Varying Directed Graphs
}

\author{C\'esar A. Uribe and Ali Jadbabaie % <-this % stops a space
	\thanks{This research was supported in part by DARPA Lagrange and a Vannevar Bush Fellowship.}
\thanks{The authors are with the Laboratory for Information and Decision Systems (LIDS), and the Institute for Data, Systems, and Society (IDSS),
        Massachusetts Institute of Technology, 77 Massachusetts Ave, Cambridge, MA 02139
        {\tt\small \{cauribe,jadbabai\}@mit.edu}}%
}

\begin{document}

\maketitle
\thispagestyle{empty}
\pagestyle{empty}

%%%%%%%%%%%%%%%%%%%%%%%%%%%%%%%%%%%%%%%%%%%%%%%%%%%%%%%%%%%%%%%%%%%%%%%%%%%%%%%%
\begin{abstract}

We study the convergence of the log-linear non-Bayesian social learning update rule, for a group of agents that collectively seek to identify a parameter that best describes a joint sequence of observations. Contrary to recent literature, we focus on the case where agents assign decaying weights to its neighbors, and the network is not connected at every time instant but over some finite time intervals. We provide a necessary and sufficient condition for the rate at which agents decrease the weights and still guarantees social learning. 

\end{abstract}

%%%%%%%%%%%%%%%%%%%%%%%%%%%%%%%%%%%%%%%%%%%%%%%%%%%%%%%%%%%%%%%%%%%%%%%%%%%%%%%%
\section{INTRODUCTION}

The theory of non-Bayesian social learning~\cite{molavi2018theory} has gained increasing attention in recent years for its ability to provide simple and practical models for inference in complex environments where a large number of decision makers repeatedly interact over some network structure. In contrast to a fully rational approach, where agents incorporate new information in a Bayesian manner, the non-Bayesian social learning model assumes agents use some other functional form to aggregate its information and construct new beliefs~\cite{jad12}. Some examples of these aggregation schemes build upon classical results on linear~\cite{deg74} and log-linear models~\cite{gil93}.

The basic non-Bayesian social learning model assumes a group of agents tries to identify the state of the world via sequentially receiving information about an unknown state, and communicating with other agents in their social clique. Moreover, agents incorporate the received information using some social learning rule~\cite{molavi2018theory,jad12}. A group is said to achieve social learning if all agents can identify the state of the world even if their local private signals do not provide sufficient information. For a summary of some of the recent results in non-Bayesian social learning see~\cite{ned16c}.

One of the enabling tools for the study of non-Bayesian social learning is the analysis of distributed averaging algorithms~\cite{ned09,tsi84,touri2014product}. Therefore, most of the existing results about the convergence of beliefs in social learning inherit some requirements about the connectivity of the network and the persistence of weights an agent assigns to its neighbors. In terms of graph connectivity, a group of agents following log-linear update rules has been shown to be able to learn the unknown state parameter for fixed undirected graphs~\cite{jad12,ned17e}, fixed directed graphs, time-varying undirected graphs~\cite{ned14}, time-varying directed graphs~\cite{ned15b}. In terms of graph connectivity, uniform connectivity has been shown sufficient for social learning~\cite{ned15b}. However, the central assumption in most of the previous results is that the weight an agent assigns to a neighbor at any time instant is lower bounded by some positive constant. This implies that, although the graph might change with time, the links are persistent and its effects do not decay with time.

The assumption of the existence of lower bounds on the weights for time-varying graphs poses some structural constraints that ease the convergence analysis of the aggregating schemes. However, recent approaches pose the question of increasing self-confidence, where an agent gradually increases its self-weights and at the same time assigns a decaying weight to its neighbors~\cite{wang2016opinion}. Such behavior is justified by the intuition that an agent might become more and more confident in its own opinion as aggregates more information~\cite{wang2016opinion}. Also, one might consider agents that assume only they are becoming more informed over time, while others are not~\cite{demarzo2003persuasion}. An increasing self-confidence implies that an agent will assign a zero weight to the opinion of its neighbors eventually. 

In~\cite{demarzo2003persuasion}, the authors provided a condition for the rate at which the weights decay such that social learning is guaranteed for fixed graphs. In~\cite{molavi2018theory} this result was extended to time-varying graphs that are connected at each time instant. In~\cite{wang2016opinion}, the authors considered a fixed rate of decay of $O(1/k)$ and showed learning is achieved for both fixed and time-varying graphs that are always connected. Other authors have considered the phenomena of asymptotic isolation of agents in a network assuming the intervals of intercommunication between them increase with time~\cite{lorenz2011convergence,olshevsky2018fully}. 

In this paper, we generalize existing results and provide the conditions for which social learning happens, even if agents have an increasing self-confidence or decaying weights for a uniform strongly connected sequences of graphs, that is, sequences of graphs that might not be connected at each time instant but over finite periods of time. Also, we show that this result is tight by providing a necessary condition for social learning, on the rate at which the weights decay. If such condition does not hold, one cannot guarantee learning occurs for general sequences of uniformly connected graphs. This result is of independent interests for the literature of average consensus and distributed optimization, as it provides a condition to guarantee that consensus is achieved even if the weight matrices do not have lower bounded non-zero entries.

This paper is organized as follows. Section~\ref{sec:problem} presents the problem of non-Bayesian social learning. Also, we recall some basic definitions and assumptions, and state our main result. Section~\ref{sec:consistency} shows the proof our main results, and introduce some auxiliary technical lemmas. Section~\ref{sec:conflict} extends the main result to the case of agents with conflicting hypotheses. Section~\ref{sec:converse} presents a converse result, where we show a necessary condition to guarantee social learning. Section~\ref{sec:numerical} presents a numerical example that validates our theoretical results. Finally, Section~\ref{sec:conclusions} presents some conclusions and future work.

\textbf{\textit{Notation}}: Random variables are represented as upper case letters, i.e. $X$, whereas their realizations as its corresponding lower case, i.e. $x$. Subscripts will denote time indices and make use of the letter $k$. Agent indices are represented as superscripts and use the letters $i$ or $j$. The $i$-th row and $j$-th column entry of a matrix $A$ is denoted as $[A]_{ij}$. Moreover, for a sequence of matrices $\{A_{k}\}$ we denote $A_{k:t} = A_kA_{k-1}\cdots A_{t+1}A_t$ for $k \geq t$. We use $a.s.$ to refer to \textit{almost sure} convergence.  

\section{PROBLEM FORMULATION AND RESULTS}\label{sec:problem}

Consider a set of $n$ agents, denoted by $V = \{1,\dots,n\}$, who seek to learn a fixed underlying and unknown state of the world by sequentially observing realizations from a set of random variables. Particularly, at each time step $k \geq 0$, each agent $i$ receives an independent (in time and among the agents) private realization of a \textit{finite} random variable $X_k^i \sim P^i$, where we assume $P^i$ has full support over the realizations of $X^i_k$. Moreover, each agent has a private parametrized family of distributions $\mathcal{P}^i_{\Theta} = \{P^i_{\theta}\}$ for $\theta \in \Theta$, where $\Theta = \{\theta_1,\dots,\theta_p\}$ is a finite set. We will generally refer to $\Theta$ as the set of hypotheses. The group objective is to find a parameter $\theta \in \Theta$ that solves the following optimization problem
\begin{align}\label{eq:main_problem}
\min_{\theta \in \Theta} \sum_{i=1}^{n} D_{KL}(P^i \| P^i_\theta),
\end{align}
where $D_{KL}(P^i \| P^i_\theta)$ is the Kullback-Leibler (KL) divergence between the distribution $P^i$ and $P^i_\theta$. In words, the group of agents tries to identify a member of their joint parameter space such that it generates a probability distribution $\prod_{i=1}^{i}P^i_\theta$ that minimizes the KL divergence with the true distribution of the observations $\prod_{i=1}^{n}P^i$. The fact that $\Theta$ is finite guarantees that a solution of $\eqref{eq:main_problem}$ always exists. We will denote as $\Theta^*$ the subset of $\Theta$ that minimizes~\eqref{eq:main_problem}. Moreover, in order to avoid trivial solutions we assume $\Theta^* \neq \Theta$.

Clearly, if for each individual agent $\Theta^*_i = \argmin_{\theta\in\Theta} D_{KL}(P^i \| P^i_\theta)$ is non-empty and has only one element that is common to all agents, each agent can solve \eqref{eq:main_problem} separately. However, we study the general case where there are identifiability limitations, e.g., the set $\Theta^*_i$ has more than one element for some $i \in V$, yet $\bigcap_{i=1}^n \Theta^*_i$ is non-empty, and $\Theta^* \in \bigcap_{i=1}^n \Theta^*_i$. Moreover, we also study the case where $\bigcap_{i=1}^n \Theta^*_i $ is empty, that is, there might be conflicting hypothesis~\cite{ned15} in the sense that the minimizer of the local functions might not be the same for all agents. 

Under these identifiability issues, the agents collaborate with each other to jointly solve problem~\eqref{eq:main_problem}. This collaboration comes in the form of exchange of information among them. We will assume that in addition to their private signals, each agent receives at each time step the beliefs of a subset of the other agents that we will call the neighbors. We define the beliefs of an agent $i$ as a probability distribution over the simplex generated by $\Theta$. We denote by $\mu^i_k(\theta)$ the belief that agent $i$ has at time $k$ about the hypothesis $\theta$. The communication among the agents is modeled as a sequence of graphs $\{\mathcal{G}_k\}$, where $\mathcal{G}_k  = (V,\mathcal{E}_k)$, and $\mathcal{E}_k$ is a set of edges such that $(j,i) \in \mathcal{E}_k$ if agent $j$ can send information to agent $i$ at time $k$.

In this paper, we study the group dynamics where each agent $i \in V$ updates its beliefs following the \textit{log-linear social learning rule}:
\begin{align}\label{eq:update}
\mu^i_{k+1} (\theta) &  = \frac{\prod_{j=1}^{n} \mu^j_k(\theta)^{[T_k]_{ij}} p^i_\theta(x^i_{k+1})}{\sum_{\hat \theta \in \Theta}\prod_{j=1}^{n} \mu^j_k(\hat \theta)^{[T_k]_{ij}} p^i_{\hat\theta}(x^i_{k+1})}, 
\end{align}
where $p^i_\theta(x^i_{k}) = P(X^i_{k} = x^i_k | \theta^* = \theta)$ denotes the probability of observing the realization $x^i_k$ at time $k$ conditioned on the true hypothesis being $\theta$. Moreover $T_k$ is a non-negative matrix of weights compliant with the underlying structure of the graph $\mathcal{G}_k$.

In contrast with other works, our objective is to establish necessary conditions for the convergence of the beliefs of all the agents in the network to the solution of~\eqref{eq:main_problem} for a weaker connectivity assumption. Mainly, we are interested in whether social learning is achieved with the log-linear update rule~\eqref{eq:update} for agents with increasing self-confidence, i.e., the weights an agent assigns to its neighbors decay to zero, or equivalently, the self-weight $[T_k]_{ii}$ converges to $1$, for all $i \in V$, as $k$ increases. In general, one should expect that the rate at which an agent increases its self-confidence should not be too fast, as not enough information from its neighbors might arrive. 

Next, we recall some basic assumptions and definitions. 

\begin{definition}
	A sequence of graphs $\{\mathcal{G}_k\}$ is uniformly strongly connected or $B$-strongly connected, if there exists an integer $B>0$ such that the graph with edge set
	\begin{align*}
	\mathcal{E}_k^B & = \bigcup_{i=kB}^{(k+1)B-1} \mathcal{E}_i
	\end{align*}
	is strongly connected for every $k\geq 0$.
\end{definition}

\begin{definition}[Definition $2.1$ in~\cite{touri2012product}]
	Let $\{A_k\}$ be a sequence of row stochastic matrices. We say that $\{A_k\}$ is ergodic if $\lim_{k\to \infty} A_k A_{k-1} \dots A_s = \boldsymbol{1}\phi_s^T$ for any $s\geq 0$, where $\phi_s$ is a stochastic vector. 
\end{definition}

\begin{assumption}\label{assum:bounds_on_a}
	Given a sequence of graphs $\{\mathcal{G}_k\}$ that is $B$-strongly connected, then
	\begin{enumerate}[(a)]
		\item For each $k$, there exists a weight matrix $T_k$ that is row-stochastic and compliant with the underlying graph topology, i.e., $[T_{k}]_{ij} > 0$ if  $(j,i) \in  \mathcal{E}_k$.
		\item There exists a sequence $\{\lambda_k\}$, with $\lambda_k \in (0,1)$, and constants $\underline{a}\in (0,1)$ and $\bar{a}\in (0,1)$, such that
		\begin{align*}
		[T_k]_{ij} & \geq \lambda_k \underline{a},\\
		\sum\limits_{j \neq i} [T_k]_{ij} & \leq \lambda_k \bar{a},
		\end{align*}
		for all $k \geq 0$ and all pairs of agents such that $(j,i)\in\mathcal{E}_k$ and $i \neq j$.
	\end{enumerate}
\end{assumption}

Assumption~\ref{assum:bounds_on_a}(a) ensures that the sequence of weights used in the update rule~\eqref{eq:update} is consistent with the structure of the communication network among the agents. Particularly, if agent $j$ can send information to agent $i$ at a time instant $k$, then the edge $(j,i) \in \mathcal{E}_k$, and therefore the agent $i$ assigns a positive weight to the information coming from agent $j$, i.e.,  $[T_k]_{ij}>0$. Assumption~\ref{assum:bounds_on_a}(b) limits the rate of decay of the weights an agent assigns to its neighbors. Particularly, even if an edge does not exist at every time step, its corresponding strength do not decay faster than the sequence $\{\lambda_k\}$. This in turn limits the rate at which the self-confidence of an agent increases. Our main result will characterize the conditions on $\{\lambda_k\}$ to guarantee the network of agents will learn the state.

\begin{assumption}\label{assum:no_conflict}
	The set $\bigcap_{i=1}^n \Theta_i^*$, where 
	$\Theta_i^* = \argmin\limits_{\theta \in \Theta} D_{KL}\left(P^i\|P^i_\theta \right)$
	for each $i$, is non-empty. Moreover, $\Theta^* \subseteq \bigcap_{i=1}^n \Theta_i^*$.
\end{assumption}

Assumption \ref{assum:no_conflict} guarantees that even if some agents cannot correctly identify $\Theta^*$,  the optimal set that solves~\eqref{eq:main_problem} lies inside the optimal set of the solution of their local problem. Later in Section~\ref{sec:conflict}, we will remove this assumption by allowing conflicting hypotheses, i.e.,  $\bigcap_{i=1}^n \Theta_i^*$ being empty.

Next, we state our main result regarding the consistency of the learning rule~\eqref{eq:update}. That is, all the agents in the network concentrate their beliefs on the set $\Theta^*$ that solves~\eqref{eq:main_problem}.

\begin{theorem}\label{thm:main}
	Let Assumptions~\ref{assum:bounds_on_a} and \ref{assum:no_conflict} hold. If
	\begin{align}\label{eq:inf_often}
	\lim_{k \to \infty } k \prod_{i=kB}^{(k+1)B-1} \lambda_i & = \infty. 
	%\sum_{k=1}^\infty \prod_{i=kB}^{(k+1)B-1} \lambda_i & = \infty\\ 
	\end{align} 
	Then, the update rule~\eqref{eq:update} has the following property:
	\begin{align*}
	\lim_{k \to \infty } \mu^i_k(\theta)=0 \qquad a.s. \qquad \forall\theta \notin \Theta^*, i \in V.
	\end{align*}
\end{theorem}

Condition~\eqref{eq:inf_often}, on the rate of decrease on $\{\lambda_k\}$, states that sequential products of the form $\prod_{i=kB}^{(k+1)B-1} \lambda_i$ should not decrease too fast. Particularly, for a connectivity parameter $B$ it is sufficient to guarantee that $\lambda_k > O(k^{-1/B})$. That is, the total decreases in the time period of size $B$ should not be faster than $O(1/k)$. Moreover, if $B=1$ we recover the same condition as in~\cite{molavi2018theory,demarzo2003persuasion}. Additionally, \eqref{eq:inf_often} implies that $\sum_{k=1}^\infty \prod_{i=kB}^{(k+1)B-1} \lambda_i = \infty$.

\section{Consistency of Social Learning with Increasing Self-Confidence}\label{sec:consistency}

In this section, we prove our main result in Theorem~\ref{thm:main}. Initially, we provide a fundamental assumption about the existence of a strictly positive lower bound on the weights a node assigns to the information coming from other nodes.

\begin{assumption}\label{assum:lower_bound}
	For each $k$, the matrix $A_k$ is stochastic, i.e., $\sum_{j=1}^n [A_k]_{ij} =1$ for $i \in V$, with positive diagonal entries. Additionally, there exists a constant $\eta>0$ such that if $[A_k]_{ij}>0$ then $[A_k]_{ij}>\eta$.
\end{assumption}

With Assumption~\ref{assum:lower_bound} at hand,  we state a well-known result about the ergodicity of the backward product of row stochastic matrices.

\begin{lemma}[Lemma $2$ in~\cite{ned13}]\label{lemma:angelia}
	Suppose that the graph sequence $\{\mathcal{G}_k\}$ is uniformly strongly connected and let Assumption~\ref{assum:lower_bound} hold. Then, for each integer $s \geq 0$, there is a stochastic vector $\phi_s$ such that for all $i,j$ and $k>s$
	\begin{align*}
	|[A_k A_{k-1} \dots A_{s+1}A_s]_{ij} -\phi^j_s  | & \leq 2\left(\left( 1 - \eta^{nB} \right)^{\frac{1}{nB}}\right)^{k-s}.
	\end{align*}
\end{lemma}

One immediate conclusion from Lemma~\ref{lemma:angelia} the sequence $\{A_k\}$ is ergodic.

The next auxiliary lemma states that the entries of the sequence of stochastic vectors $\{\phi_k\}$ are lower bounded by a strictly positive value that depends on the structure of $\{\mathcal{G}_k\}$.

\begin{lemma}[Lemma~$4$ and Corollary~$2$ in~\cite{ned13}]\label{lemma:angelia2}
	Given a graph sequence $\{\mathcal{G}_k\}$, define
	\begin{align*}
	\delta & \triangleq \inf_{k \geq 0} \left(\min_{i \in V}[ \boldsymbol{1}^T A_kA_{k-1}\dots A_0 ]_i\right).
	\end{align*}
	If the graph sequence $\{\mathcal{G}_k\}$ is $B$-strongly connected, then $\delta \geq \eta^{nB}$. Moreover, if $A_k$ is doubly stochastic for all $k \geq 0 $ or if $\{\mathcal{G}_k\}$ is regular, then $\delta = 1$. Furthermore, $\phi_s$ as defined in Lemma~\ref{lemma:angelia} satisfies $\phi^i_k \geq \delta/n$ for all $k\geq 0 $ and $i \in V$.
\end{lemma}

Next, we state our first auxiliary result that will be fundamental to the construction of the proof of Theorem~\ref{thm:main}. Particularly, we provide a condition on the sequence $\{\lambda_k\}$ such that the $\{T_k \}$, for which Assumption~\ref{assum:bounds_on_a} holds, is ergodic. The next lemma will limit the rate at which the self-confidence increases such that we can guarantee there is sufficient mixing among the nodes in the network.

\begin{lemma}\label{lemma:inf_often}
	Suppose Assumption~\ref{assum:bounds_on_a} holds for a $B$-strongly connected sequence of graphs $\{\mathcal{G}_k\}$. If~\eqref{eq:inf_often} holds. Then, $\{T_k\}$ is ergodic.
\end{lemma}

\begin{proof}
	Initially, it follows from Assumption~\ref{assum:bounds_on_a}, that each member of the sequence of weight matrices $\{T_k\}$ can be written as
	\begin{align*}
	T_k & = (1 - \lambda_k) I + \lambda_k A_k,
	\end{align*}
	where $\{A_k\}$ is a sequence of stochastic matrices whose nonzero elements are uniformly lowered bounded by \mbox{$\eta = \min\{\underline{a},\bar{a}\} \in  (0,1)$}, i.e., $[A_k]_{ij} \geq \eta$ for all $(j,i)\in \mathcal{E}_k$.
	
	Following the approach proposed in~\cite{demarzo2003persuasion}, we define $\{\Lambda_k\}$ as a sequence of independent Bernoulli random variables where $P(\Lambda_k = 1) = \lambda_k $, and   $P(\Lambda_k = 0) = 1-\lambda_k $. Therefore, we can write each of the elements of the sequence of matrices $\{T_k\}$ as
	\begin{align*}
	T_k &= \mathbb{E}[(1-\Lambda_k)I + \Lambda_k A_k],
	\end{align*}
	where the expectation is taken with respect to the random variable $\Lambda_k$. Moreover, define the random matrix $Z_k$ by
	\begin{align*}
	Z_k & = \prod_{t=s}^{k} [(1-\Lambda_t)I + \Lambda_t A_t]  = \prod_{t=s}^{k} A_t^{\Lambda_t}. 
	\end{align*}
	
	Thus,
	\begin{align*}
	\mathbb{E}[Z_k] &=\prod_{t=s}^k T_t  = \prod_{t=s}^k \mathbb{E}[(1-\Lambda_k)I + \Lambda_k A_k].
	\end{align*}
	
	Now, define the sequence of random variables $\{\beta_k\}$ by
	\begin{align}\label{eq:betas}
	\beta_k & = \begin{cases}
	1 &\text{if  $\sum_{i=kB}^{(k+1)B-1} \Lambda_i = B$},\\
	0 &\text{otherwise}.
	\end{cases} 
	\end{align}
	
	The random variable $\beta_k$ serves as an indicator function for the event $\{\Lambda_i = 1 \ \forall i \in (kB, \dots (k+1)B-1) \}$. Particularly, if $\{\beta_k =1\}$, then the product $\prod_{i=kB}^{(k+1)B-1} A_i^{\Lambda_t} = \prod_{i=kB}^{(k+1)B-1} A_i$. Moreover, if the event $\{\beta_k = 1\}$ occurs infinitely many times, then from Lemma~\ref{lemma:angelia} it follows that
	\begin{align}\label{eq:inf_prod_occurs}
	\lim_{k \to \infty } \mathbb{E}[Z_k] &= \lim_{k \to \infty } \prod_{t=s}^k T_t =  \boldsymbol{1}\phi_s^T.
	\end{align}
	
	Therefore, to complete the proof we need to show that if~\eqref{eq:inf_often} holds, then the event $\{\beta_k = 1\}$ occurs infinitely often. Initially, note that 
	\begin{align*}
	P(\beta_k =1) & = \prod_{i=kB}^{(k+1)B-1} P(\Lambda_i = 1) \\
	& = \prod_{i=kB}^{(k+1)B-1} \lambda_i.
	\end{align*}
	
	Thus, by the Borel-Cantelli lemma, if \mbox{$\sum_{k=1}^\infty  P(\beta_k=1)  = \infty$}, then \mbox{$P(\{\beta_k=1\} \ \text{infinitely often} ) = 1$}.
	
	Moreover, if~\ref{eq:inf_often} holds then
	\begin{align*}
	\sum_{k=1}^\infty  P(\beta_k=1)  = \sum_{k=1}^\infty \prod_{i=kB}^{(k+1)B-1} \lambda_i = \infty,
	\end{align*}
	and the desired result holds.
\end{proof}

Lemma~\ref{lemma:inf_often} shows that if~\eqref{eq:inf_often} holds, then $\{T_k\}$ is ergodic. That is, even if the weights, an agent assigns to its neighbors, decays to zero, if the rate at which that happens is sufficiently slow, then the resulting infinite product is equivalent to an infinite product of row stochastic matrices with lower bounded entries.  

The next lemma states the existence of an absolute probability sequence for the Markov chain generated by $\{T_k\}$~\cite{touri2012product}. Later in the proof of our main theorem, we will make use of this absolute probability sequence.

\begin{lemma}
	There exists an absolute probability $\{\phi_k\}$ sequence for the chain $\{T_k\}$, i.e., 
	\begin{align*}
	\phi_{k+1}^T T_k T_{k-1}\dots T_s & = \phi_s^T.
	\end{align*} 
\end{lemma}

\begin{proof}
	This result follows immediately from Lemma~\ref{lemma:inf_often} and Lemma~$4.4$ in~\cite{touri2012product}.    
\end{proof}

Now, we are ready to prove our main result regarding the conditions for a group of agents with increasing self-confidence to reach social learning.

\begin{proof}[Theorem~\ref{thm:main}]
	Initially, define the random variables
	\begin{align*}
	\varphi^i_k(\theta,\theta^*) & = \log \frac{\mu^i_k(\theta)}{\mu^i_k(\theta^*)}, \text{ and } 
	\mathcal{L}^i_{\theta,\theta^*}(X^i_{k} )   = \log \frac{p^i_\theta(X^i_{k})}{p^i_{\theta^*}(X^i_{k})}.
	\end{align*}
	
	Thus, it follows form the update rule~\eqref{eq:update} that
	\begin{align*}
	\varphi^i_{k+1}(\theta,\theta^*)  & = \sum_{j=1}^{n}[T_k]_{ij} \varphi^j_k(\theta,\theta^*) + \log \frac{p^i_\theta(X^i_{k+1})}{p^i_{\theta^*}(X^i_{k+1})},
	\end{align*}
	or equivalently, by stacking all entries of $\varphi^i_{k+1}(\theta,\theta^*)$ and $\mathcal{L}^i_{\theta,\theta^*}(X^i_{k} )$ into single column vectors $\varphi_{k+1}(\theta,\theta^*)$ and $\mathcal{L}_{\theta,\theta^*}(X^i_{k} )$, where $[\varphi_{k+1}(\theta,\theta^*)]_i = \varphi^i_{k+1}(\theta,\theta^*)$ and $[\mathcal{L}_{\theta,\theta^*}(X^i_{k} ) ]_i = \mathcal{L}^i_{\theta,\theta^*}(X^i_{k} ) $,
	\begin{align}\label{eq:phi_vector}
	\varphi_{k+1}(\theta,\theta^*) & = T_k \varphi_k(\theta,\theta^*) + \mathcal{L}_{\theta,\theta^*}(X^i_{k+1} ) \\
	\varphi_{k}(\theta,\theta^*)& =  \sum_{t=1}^{k-1} T_{k-1:t}\mathcal{L}_{\theta,\theta^*}(X^i_{t} ) + \mathcal{L}_{\theta,\theta^*}(X^i_{k} ),
	\end{align}
	where we have assumed without loss of generality that $\varphi^i_0(\theta,\theta^*) =0$ for all $\theta \in \Theta$ and all $i\in V$, which is equivalent to all agents having uniform beliefs at time $k=0$.
	
	To complete the proof, we first show that $\limsup_{k \to \infty} \frac{1}{k} \sum_{i=1}^{n} \phi^i_k \varphi^i_k(\theta,\theta^*) <0$. This will imply that the weighted sum of belief the beliefs of all neighbors on the wrong hypothesis will asymptotically converge to zero.
	
	If we pre-multiply~\eqref{eq:phi_vector} by the absolute probability sequence $\phi_k$ from Lemma~\ref{lemma:inf_often}, we have that
	\begin{align*}
	\sum_{i=1}^n \phi_k^i\varphi^i_{k}(\theta,\theta^*) 
	& =  \sum_{t=1}^{k} \sum_{i=1}^n \phi_t^i \mathcal{L}^i_{\theta,\theta^*}(X^i_{t} ) .    
	\end{align*}
	
	Moreover, by adding and subtracting $\mathbb{E}[\mathcal{L}^i_{\theta,\theta^*}(x^i_{t} )]$, it follows that
	\begin{align}\label{eq:add_sub}
	\sum_{i=1}^n \phi_k^i\varphi^i_{k}(\theta,\theta^*) & =  \sum_{t=1}^{k} \sum_{i=1}^n \phi_t^i \left( \mathcal{L}^i_{\theta,\theta^*}(X^i_{t} ) - \mathbb{E}[\mathcal{L}^i_{\theta,\theta^*}(X^i_{t} )]\right) + \nonumber \\
	& \qquad +\sum_{t=1}^{k} \sum_{i=1}^n \phi_t^i \mathbb{E}[\mathcal{L}^i_{\theta,\theta^*}(X^i_{t} )].
	\end{align}
	
	Furthermore, by dividing by $k$ on both sides and taking the limit supremum as $k \to \infty$, the first term on the right in~\eqref{eq:add_sub} goes to zero almost surely by the strong law of large numbers, thus
	\begin{align}\label{eq:asymptotic_rate}
	&\limsup_{k \to \infty} \frac{1}{k}\sum_{i=1}^n \phi_k^i\varphi^i_{k}(\theta,\theta^*) \nonumber\\
	& =  \limsup_{k \to \infty}  \frac{1}{k} \sum_{t=1}^{k} \sum_{i=1}^n \phi_t^i \mathbb{E}[\mathcal{L}^i_{\theta,\theta^*}(X^i_{t} )] \nonumber\\
	& = \sum_{i=1}^{n}  \phi_t^i \left(D_{KL}(P^i\|P^i_{\theta^*}) -D_{KL}(P^i\|P^i_{\theta}) \right) \nonumber \\
	& = \frac{\delta}{n}\sum_{i=1}^{n}  \left(D_{KL}(P^i\|P^i_{\theta^*}) -D_{KL}(P^i\|P^i_{\theta}) \right) \nonumber\\
	& < 0,
	\end{align}
	where we have used Lemma~\ref{lemma:angelia2} and the fact that $\mathbb{E}[\mathcal{L}^i_{\theta,\theta^*}(X^i_{t} )] = D_{KL}(P^i\|P^i_{\theta^*}) -D_{KL}(P^i\|P^i_{\theta})$. Moreover, by Assumption~\ref{assum:no_conflict} it holds that $D_{KL}(P^i\|P^i_{\theta^*}) < D_{KL}(P^i\|P^i_{\theta})$.
	
	The relation in~\eqref{eq:asymptotic_rate} shows that for at least one of the agents, the beliefs on the non-optimal hypotheses $\theta\notin \Theta^*$ will asymptotically go to zero. To complete the proof, we proceed to show that the difference between the beliefs among the agents decays to zero as well, which in turns implies that all agents eventually assign a zero belief to the non-optimal hypotheses.
	
	Now, following the same approach as in~\cite{molavi2018theory}, we proceed to bound the asymptotic difference between the logarithmic ratio of beliefs among the two agents with the most separate beliefs. Initially we have that,
	\begin{align*}
	&\max_{i\in V} \varphi^i_k(\theta,\theta^*) - \min_{i\in V} \varphi^i_k(\theta,\theta^*)  \\
	&\qquad\leq   \max_{i\in V} \mathcal{L}^i_{\theta,\theta^*}(X^i_{k} ) - \min_{i\in V} \mathcal{L}^i_{\theta,\theta^*}(X^i_{k} ) +\\
	&\qquad \qquad +  \sum_{t=1}^{k-1}\pi(T_{k-1}\dots T_{t+1}T_{t}) \times \\
	&\qquad  \qquad \times \left(\max_{i\in V} \mathcal{L}^i_{\theta,\theta^*}(X^i_{t} ) - \min_{i\in V} \mathcal{L}^i_{\theta,\theta^*}(X^i_{t} )\right) \\
	& \qquad \leq \sum_{t=1}^{k}\pi\left(\mathbb{E}[Z_t]\right)\left(\max_{i\in V} \mathcal{L}^i_{\theta,\theta^*}(X^i_{t} ) - \min_{i\in V} \mathcal{L}^i_{\theta,\theta^*}(X^i_{t} )\right),
	\end{align*}
	where the function $\pi(A)\in [0,1]$ is defined as
	\begin{align*}
	\pi(A) & = 1 - \min_{i,j \in V} \sum_{l=1}^{n} \min\{[A]_{ik},[A]_{jk}\} \\
	& \leq 1 - \max_{i\in V} \min_{i \in V} [A]_{ij},
	\end{align*}  
	and is convex and sub-multiplicative, see Lemma~$A.2$ in~\cite{molavi2018theory}.
	
	Under the assumption that the variables $X^i_k$ are finite for all $i \in V$ and $k \geq 0$, it follows that there exists a constant $c\geq 0$, independent of $k$, such that    
	\begin{align}\label{eq:proof_summ}
	\max_{i\in V} \varphi^i_k(\theta,\theta^*) - \min_{i\in V} \varphi^i_k(\theta,\theta^*) 
	& \leq c\sum_{t=1}^{k}\mathbb{E}\left[\pi\left(Z_t\right)\right].
	\end{align} 
	
	Next, without loss of generality we assume the value $k$ is such that we can write it as \mbox{$k = (s+1)nB-1$} for some $s\geq 0$. This will allow us to group the summation on the right side of~\eqref{eq:proof_summ} into some initial finite sum and sets of size $nB$ as follows     
	\begin{align*}
	&\max_{i\in V} \varphi^i_k(\theta,\theta^*) - \min_{i\in V} \varphi^i_k(\theta,\theta^*)  \\
	& \qquad \leq c\sum_{t=1}^{s}\sum_{i=tnB}^{(t+1)nB-1}\mathbb{E}\left[\pi\left(Z_i\right)\right] + c\sum_{t=1}^{nB-1}\mathbb{E}\left[\pi\left(Z_t\right)\right].
	\end{align*}
	
	As next step, we use Lemma~\ref{lemma:angelia} to bound the entries of $Z_i$ and count how many events $\{\beta_k =1\}$ occur. Particularly, we know that if $\{\beta_k =1\}$ occurs $nB$ times, the smallest entry of $Z_i$ will be lower bounded by $\eta^{nB}$. Thus
	\begin{align*}
	&\max_{i\in V} \varphi^i_k(\theta,\theta^*) - \min_{i\in V} \varphi^i_k(\theta,\theta^*)  \\
	& \leq c\sum_{t=1}^{s}\mathbb{E}\left[\left(1 - \eta^{nB}\right)^{\lfloor \frac{ \beta_t + \dots + \beta_s  }{nB}\rfloor}  \right] + c\sum_{t=1}^{nB-1}\mathbb{E}\left[\pi\left(Z_t\right)\right]\\
	& \leq c\sum_{t=1}^{s}\mathbb{E}\left[2\left(\left(1 - \eta^{nB}\right)^{\frac{1}{nB}}\right)^{ \beta_t + \dots + \beta_s }\right] + c\sum_{t=1}^{nB-1}\mathbb{E}\left[\pi\left(Z_t\right)\right]\\
	& \leq 2c\sum_{t=1}^{s}\mathbb{E}\left[\sigma^{\beta_t + \dots + \beta_s }\right] + c\sum_{t=1}^{nB-1}\mathbb{E}\left[\pi\left(Z_t\right)\right],
	\end{align*}
	where we have defined $\sigma = \left(1 - \eta^{nB}\right)^{\frac{1}{nB}}$.
	
	Using the fact that the random variables $\{\beta_k\}$ are independent, we have that    
	\begin{align*}
	&\max_{i\in V} \varphi^i_k(\theta,\theta^*) - \min_{i\in V} \varphi^i_k(\theta,\theta^*)  \\
	& \leq 2c\sum_{t=1}^{s} \prod_{r=t}^s \left(1 - (1-\sigma) P(\beta_r = 1) \right) + c\sum_{t=1}^{nB-1}\mathbb{E}\left[\pi\left(Z_t\right)\right]\\
	& \leq 2c\sum_{t=1}^{s} \prod_{r=s}^s \left(1 - (1-\sigma)\prod_{i=rB}^{(r+1)B-1}\lambda_i \right) + c\sum_{t=1}^{nB-1}\mathbb{E}\left[\pi\left(Z_t\right)\right].
	\end{align*}
	
	Now, define $\alpha_s = \min_{1\leq r \leq s } \prod_{i=rB}^{(r+1)B-1}\lambda_i$, then
	\begin{align}\label{eq:proof_end}
	&\max_{i\in V} \varphi^i_k(\theta,\theta^*) - \min_{i\in V} \varphi^i_k(\theta,\theta^*)  \nonumber\\
	& \leq 2c\sum_{t=1}^{s} \left(1 - (1-\sigma) \alpha_s \right)^{s-t} + c\sum_{t=1}^{nB-1}\mathbb{E}\left[\pi\left(Z_t\right)\right]  \nonumber\\
	& \leq 2c\frac{1 - \left(1-(1-\sigma) \alpha_s \right)^s}{(1 - \sigma)\alpha_s }  + c\sum_{t=1}^{nB-1}\mathbb{E}\left[\pi\left(Z_t\right)\right],
	\end{align}
	
	Finally, divide both sides of~\eqref{eq:proof_end} by $k$ and take $\limsup$ as $k \to \infty$, then
	\begin{align}\label{eq:proof_final}
	&\limsup_{k \to \infty} \frac{1}{k} \left(\max_{i\in V} \varphi^i_k(\theta,\theta^*) - \min_{i\in V} \varphi^i_k(\theta,\theta^*) \right)  \nonumber \\ 
	& \leq \limsup_{k \to \infty} \frac{1}{k} 2c\frac{1 - \left(1-(1-\sigma) \alpha_k \right)^k}{(1 - \sigma)\alpha_k }  \nonumber \\
	& \leq 0,
	\end{align}
	where the last line follows from~\eqref{eq:inf_often} and the fact that the last term on the right of~\eqref{eq:proof_end} is finite. Thus,
	\begin{align*}
	\lim_{k \to \infty } \frac{1}{k}\left(\varphi^i_k(\theta,\theta^*) - \varphi^j_k(\theta,\theta^*)\right) = 0 \qquad a.s.
	\end{align*}
	
	The desired result will follow from~\eqref{eq:asymptotic_rate} and~\eqref{eq:proof_final} since they imply that $\limsup_{k \to \infty} \frac{1}{k} \varphi^i_{k}(\theta,\theta^*) <0$ almost surely for all agents. Therefore, $\lim_{k \to \infty } \varphi^i_{k+1}(\theta,\theta^*) = -\infty$, which subsequently imply that $\lim_{k \to \infty } \mu^i_k(\theta) = 0$ for all $\theta \notin \Theta^*$ with probability $1$.  
\end{proof}

\section{A CONVERSE RESULT FOR THEOREM~\ref{thm:main}}\label{sec:converse}

In this section, we state an additional result regarding the ergodicity of the backward product of the matrices from the sequence $\{T_k\}$. Particularly, Lemma~\ref{lemma:inf_often} shows that if~\eqref{eq:inf_often} holds then~$\{T_k\}$ is ergodic. This indicates that if the weight an agent assigns to its neighbors decays sufficiently slow, the infinite backward product of its weight matrices is equivalent to an infinite product of a sequence of matrices whose positive entries are lower bounded. Now, we explore the case when~\eqref{eq:inf_often} does not hold.  

\begin{lemma}\label{lemma:converse}
	Suppose Assumption~\ref{assum:bounds_on_a} holds for a $B$-strongly connected sequence of graphs $\{\mathcal{G}_k\}$. If
	\begin{align}\label{eq:converse}
	\lim_{k \to \infty } k \prod_{i=kB}^{(k+1)B-1} \lambda_i & < \infty. 
	%\sum_{k=1}^\infty \prod_{i=kB}^{(k+1)B-1} \lambda_i & = \infty\\ 
	\end{align} 
	Then, $P(\{\beta_k =1 \ \text{infinitely often} \}) =0$.
\end{lemma}

\begin{proof}
	The desired result follows from the same argument as the proof of Lemma~\ref{lemma:inf_often}. By the Borel-Cantelli lemma, if
	\begin{align}
	\sum_{k=1}^\infty  P(\beta_k=1) = \sum_{k=1}^\infty \prod_{i=kB}^{(k+1)B-1} \lambda_i & < \infty.
	\end{align}
	Then, $P(\{\beta_k=1\} \ \text{infinitely often} ) = 0$. Therefore, the infinite backward product of the sequence of stochastic matrices $\{T_k\}$ corresponds to the product of a finite number of matrices of the form $\prod_{i=kB}^{(k+1)B-1}A_i$, which is not sufficient for ergodicity.
\end{proof}

Lemma~\ref{lemma:converse} shows that if the sequence $\{\lambda_k\}$ decays to zero too fast, $\{T_k\}$ can be non-ergodic since it is only equivalent to a finite number of products of matrices with lower bounded positive entries, which might not result in a rank one matrix for general graphs.

The next corollary presents a consequence of Lemma~\ref{lemma:converse} in the context of non-Bayesian social learning.

\begin{corollary}\label{thm:weak_coro}
	Let Assumptions~\ref{assum:bounds_on_a} and \ref{assum:no_conflict} hold. For a group of agents following the update rule~\eqref{eq:update}. If \eqref{eq:converse} holds. Then, there exists a sufficiently large graph such that a subset of agents remains uncertain about the state of the world almost surely.
\end{corollary}

\begin{proof}
	If~\eqref{eq:converse} holds we know that the infinite product $\{T_k\}$ corresponds to a finite product of matrices with lower bounded entries. Define the number of finite products as $N$. Assume there are $2N$ agents, and they are connected over a path graph. Moreover, only one of the agents at the end of the path has informative signals. Then, 
	$    \lim_{k=\infty} \prod_{t=s}^k T_t     $, will have zero entries for some points, i.e., not enough mixing happens. Furthermore, a subset of agents will not learn.
\end{proof}

\section{SOCIAL LEARNING WITH CONFLICTING HYPOTHESES}\label{sec:conflict}

Assumption~\ref{assum:no_conflict} is central for the results presented in Theorem~\ref{thm:main}. In this section we explore the consistency of the log-linear learning rule~\eqref{eq:update} when Assumption~\ref{assum:no_conflict} does not hold. Particularly, the fact that the optimal set $\Theta^*$ is contained in the optimal set of the individual local functions, guarantees that the distance between any hypothesis $\theta \notin \Theta^*$ to the true distribution of the observations $P^i$ is larger than the distance (note that the KL divergence is only a premetric) between the hypotheses $\theta^*$ and $P^*$, i.e.,
\begin{align}\label{eq:no_conflict}
D_{KL}(P^i\|P^i_\theta) > D_{KL}(P^i\|P^i_{\theta^*}),
\end{align}
which in turn makes~\eqref{eq:asymptotic_rate} hold. 

If Assumption~\ref{assum:no_conflict} does not hold, one cannot guarantee~\eqref{eq:no_conflict} holds, because in general we only have $\sum_{i=1}^{n} D_{KL}(P^i\|P^i_\theta) \geq \sum_{i=1}^{n} D_{KL}(P^i\|P^i_{\theta^*})$, and some of the terms in the sum might be negative. Moreover, each term will be multiplied (weighted) by the corresponding entry of the vector $\phi_t$ and thus~\eqref{eq:asymptotic_rate} will depend on the specific sequence of graphs used.

One way to avoid this issue is to assume that every matrix in the sequence $\{T_k\}$ is doubly stochastic, which guarantees, by Lemma~\ref{lemma:angelia2}, that $\phi_t = 1/n$, making  ~\eqref{eq:asymptotic_rate} hold. However, in general, this approach has two main issues. First, if the graph is assumed directed, the agents might not be able to compute a set of doubly stochastic weights distributedly. Moreover, not every directed graph allows a doubly stochastic set of weights~\cite{gharesifard2010does}.

In~\cite{ned15b}, the authors proposed a modified log-linear update~\eqref{eq:update_conflict}, and showed that it guarantees all nodes in the network will correctly learn the solution of~\eqref{eq:main_problem} even in the present of conflicting hypotheses.
\begin{subequations}\label{eq:update_conflict}
	\begin{align}
	y^i_{k+1} & =\sum\limits_{j \in N_k^i }[T_k]_{ij}y_k^j \\
	\mu_{k+1}^i\left(\theta\right) & = \frac{1}{Z_{k+1}^i}\left(\prod\limits_{j =1 }^n\mu_{k}^j\left(\theta\right)^{[T_k]_{ij}y_k^j}p^i_\theta(x^i_{k+1})\right)^{\frac{1}{y_{k+1}^i}}
	\end{align}
\end{subequations}
where $d_k^i$ is the out degree of node $i$ at time $k$, and $Z^i_{k+1}$ is the corresponding normalization factor.

The fundamental result in Lemma~\ref{lemma:inf_often} extends to the update rule~\eqref{eq:update_conflict}, which directly allow us to state the following result.

\begin{theorem}\label{thm:conflict}
	Let Assumption~\ref{assum:bounds_on_a}(b) hold, and for each $k$, assume there exists a weight matrix $T_k$ that is \textit{column}-stochastic and compliant with the underlying graph topology, i.e., $[T_{k}]_{ij} > 0$ if  $(j,i) \in  \mathcal{E}_k$. If
	\begin{align*}
	\lim_{k \to \infty } k \prod_{i=kB}^{(k+1)B-1} \lambda_i & = \infty. 
	%\sum_{k=1}^\infty \prod_{i=kB}^{(k+1)B-1} \lambda_i & = \infty\\ 
	\end{align*} 
	Then, the update rule~\eqref{eq:update_conflict}, with $y_0^i=1$, has the following property:
	\begin{align*}
	\lim_{k \to \infty } \mu^i_k(\theta) = 0 \qquad a.s. \qquad \forall\theta \notin \Theta^*, i \in V.
	\end{align*}
\end{theorem}

The proof of Theorem~\ref{thm:conflict} follows similar arguments as in the proof of Theorem~\ref{thm:main}, see also~\cite{ned15b}. We omit the proof due to space constraints.

%\section{INCREASING SELF-CONFIDENCE AND GROWING INTERCOMMUNICATION INTERVALS}

\section{NUMERICAL ANALYSIS}\label{sec:numerical}

In this section, we present simulation results for the non-Bayesian social learning model under different variations of graph connectivity. We assume there is a group of $10$ agents, from which only one of them, agent $1$, receives informative signals from a random variable $X^1_k \sim \text{Bernoulli}(0.7)$. All agents have a parametrized family of distributions $\mathcal{P}_{\Theta} = \{\theta_1, \theta_2\}$, where $P^i_{\theta_1} = \text{Bernoulli}(0.2)$ and $P^i_{\theta_2} = \text{Bernoulli}(0.8)$, thus, $\Theta^* = \theta_2$. For the network model we assume that $\mathcal{G}_k$ is a path graph if $\text{mod}(k,3) =0$, and $\mathcal{G}_k$ is the completely disconnected graph otherwise, see Figure~\ref{fig:graph}. Therefore, every $3$ time steps, the graph is a path graph and otherwise, the nodes remain disconnected. We will simulate three different scenarios $\lambda_k = 0.5$, $\lambda_k = 1/k$, and $\lambda_k = 1/k^{1/3}$.

\begin{figure}
	\centering
	\begin{overpic}[width=0.4\textwidth]{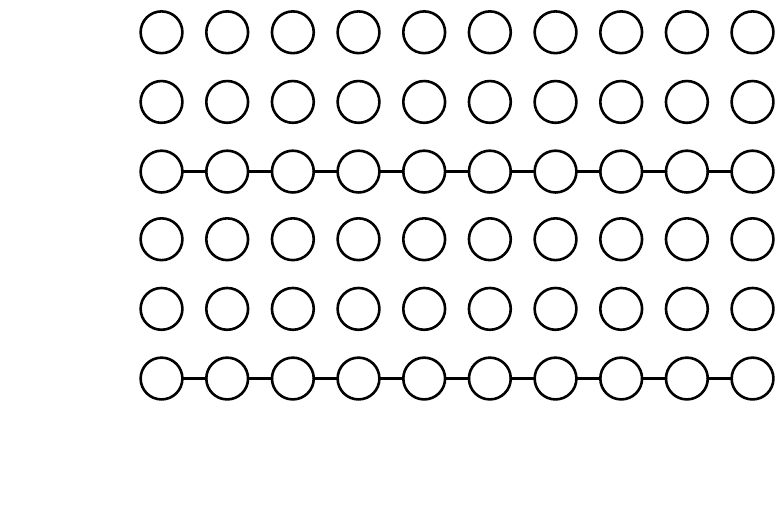}
		\put(0,61){{\small $k=1$}}
		\put(0,52){{\small $k=2$}}
				\put(0,43){{\small $k=3$}}
		\put(0,34){{\small $k=4$}}
				\put(0,25){{\small $k=5$}}
		\put(0,16){{\small $k=6$}}
		\put(5,3){{\small $\vdots$}}
		\put(60,3){{\small $\vdots$}}
	\end{overpic} 
	\caption[]{A graph sequence that is a path every $3$ iterations.}
	\label{fig:graph}
\end{figure}
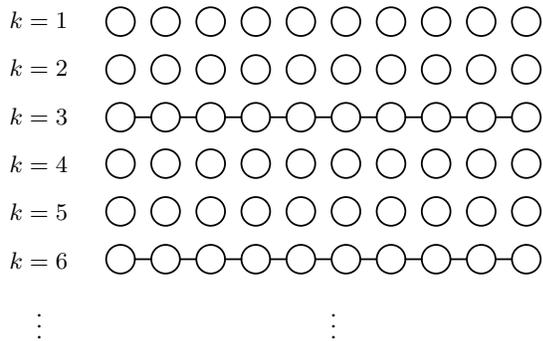

Figure~\ref{fig:s2} shows the effect of the rate at which $\lambda_k$ decays to zero. For a subset of $5$ agents, we show their beliefs on both hypotheses $\theta_1$ and $\theta_2$. Agent $1$, which is the only one with informative signals is plotted with blue color. Agent $10$ is plotted with red color. If $\{\lambda_k=0.5\}$, there is no decay and the existing results in non-Bayesian learning guarantee that the learning rate will be geometric. This can be seen since all nodes in the network concentrate their beliefs on $\theta_2$, the optimal point, at around $100$ iterations. When $\lambda_k = 1/k^{1/3}$, our results still guarantee convergence. Even though the convergence rate is slower, all nodes learn the correct state in around $1000$ iterations. Finally, when $\lambda_k = 1/k$, the sequence $\{T_k\}$ is not ergodic. Particularly, one can see that even after $1\cdot 10^7$ iterations, the red node still has not learned the state of the world.

\begin{figure}[tp]
	\centering
	%     \resizebox{6cm}{!}{
	% \begin{tikzpicture}
	% \draw (0,0) -- (6,0) -- (6,0.5) -- (0,0.5) -- (0,0);
	% \draw[blue,line width=3pt] (0.25,0.25) -- (0.75,0.25) node[black] at (1.3,0.25) {$s=1$};
	% \draw[red,line width=3pt] (2.3,0.25) -- (2.8,0.25) node[black] at (3.2,0.25) {$s=2$};
	% \draw[green,line width=3pt] (4.3,0.25) -- (4.8,0.25) node[black] at (5.3,0.25) {$s=4$};
	% \end{tikzpicture} } \\
		\begin{tikzpicture}
		%        \node at (1.5,3.3) { $F(x_k) - F^*$};
%		\node at (0.7,1.5) {{\footnotesize {\color{blue} \textbf{--}} $s=1$}};
%		\node at (0.7,1.2) {{\footnotesize {\color{red} \textbf{--}} $s=2$}};
%		\node at (0.7,0.9) {{\footnotesize {\color{green} \textbf{--}} $s=4$}};
		\begin{axis}[
		width=4.4cm,height=3.5cm,scale=1,
		x label style={at={(axis description cs:0.5,0.05)},anchor=north},
		y label style={at={(axis description cs:0.15,.5)},anchor=south},
		title = {$\mu^i_k(\theta_1)$},
		ylabel={$\lambda_k = 0.5$},
		ticklabel style = {font=\tiny},
%		ymode = log,
		xmode = log,
		ymin = 0, ymax=1,
		%    each nth point=500, filter discard warning=false, unbounded coords=discard ,
		xmin = 1e0, xmax=100, 
		every axis plot/.append style={line width=1pt}],
		legend pos=south west;
		\addplot [blue]         table [x index=20,y index=0]{no_decay.dat};
		\addplot [green]        table [x index=20,y index=2]{no_decay.dat};
		\addplot [black]    	 table [x index=20,y index=4]{no_decay.dat};
		\addplot [yellow]        table [x index=20,y index=6]{no_decay.dat};
		\addplot [red]         table [x index=20,y index=8]{no_decay.dat};
		\end{axis}
		\end{tikzpicture}
		\begin{tikzpicture}
		%        \node at (1.5,3.3) { $F(x_k) - F^*$};
		\begin{axis}[
		width=4.4cm,height=3.5cm,scale=1,
		x label style={at={(axis description cs:0.5,0.05)},anchor=north},
		y label style={at={(axis description cs:0.15,.5)},anchor=south},
			title = {$\mu^i_k(\theta_2)$},
		ticklabel style = {font=\tiny},
%		ymode = log,
		xmode = log,
		ymin = 0, ymax=1,
		%    each nth point=500, filter discard warning=false, unbounded coords=discard ,
		xmin = 1e0, xmax=100, 
		every axis plot/.append style={line width=1pt}],
		legend pos=south west;
		\addplot [blue]         table [x index=20,y index=10]{no_decay.dat};
\addplot [green]        table [x index=20,y index=12]{no_decay.dat};
\addplot [black]    	 table [x index=20,y index=14]{no_decay.dat};
\addplot [yellow]        table [x index=20,y index=16]{no_decay.dat};
\addplot [red]         table [x index=20,y index=18]{no_decay.dat};
		\end{axis}
		\end{tikzpicture}
	\begin{tikzpicture}
	%        \node at (1.5,3.3) { $F(x_k) - F^*$};
	%		\node at (0.7,1.5) {{\footnotesize {\color{blue} \textbf{--}} $s=1$}};
	%		\node at (0.7,1.2) {{\footnotesize {\color{red} \textbf{--}} $s=2$}};
	%		\node at (0.7,0.9) {{\footnotesize {\color{green} \textbf{--}} $s=4$}};
	\begin{axis}[
	width=4.4cm,height=3.5cm,scale=1,
	x label style={at={(axis description cs:0.5,0.05)},anchor=north},
	y label style={at={(axis description cs:0.15,.5)},anchor=south},
	ylabel={$\lambda_k = 1/ k^{1/3}$},
	ticklabel style = {font=\tiny},
	%		ymode = log,
	xmode = log,
	ymin = 0, ymax=1,
	%    each nth point=500, filter discard warning=false, unbounded coords=discard ,
	xmin = 1e0, xmax=1e7, 
	every axis plot/.append style={line width=1pt}],
	legend pos=south west;
	\addplot [blue]         table [x index=20,y index=0]{decay_1k3.dat};
	\addplot [green]        table [x index=20,y index=2]{decay_1k3.dat};
	\addplot [black]    	 table [x index=20,y index=4]{decay_1k3.dat};
	\addplot [yellow]        table [x index=20,y index=6]{decay_1k3.dat};
	\addplot [red]         table [x index=20,y index=8]{decay_1k3.dat};
	\end{axis}
	\end{tikzpicture}
	\begin{tikzpicture}
	%        \node at (1.5,3.3) { $F(x_k) - F^*$};
	\begin{axis}[
	width=4.4cm,height=3.5cm,scale=1,
	x label style={at={(axis description cs:0.5,0.05)},anchor=north},
	y label style={at={(axis description cs:0.15,.5)},anchor=south},
	ticklabel style = {font=\tiny},
	%		ymode = log,
	xmode = log,
	ymin = 0, ymax=1,
	%    each nth point=500, filter discard warning=false, unbounded coords=discard ,
	xmin = 1e0, xmax=1e7, 
	every axis plot/.append style={line width=1pt}],
	legend pos=south west;
	\addplot [blue]         table [x index=20,y index=10]{decay_1k3.dat};
	\addplot [green]        table [x index=20,y index=12]{decay_1k3.dat};
	\addplot [black]    	 table [x index=20,y index=14]{decay_1k3.dat};
	\addplot [yellow]        table [x index=20,y index=16]{decay_1k3.dat};
	\addplot [red]         table [x index=20,y index=18]{decay_1k3.dat};
	\end{axis}
	\end{tikzpicture}

	\begin{tikzpicture}
	%        \node at (1.5,3.3) { $F(x_k) - F^*$};
	%		\node at (0.7,1.5) {{\footnotesize {\color{blue} \textbf{--}} $s=1$}};
	%		\node at (0.7,1.2) {{\footnotesize {\color{red} \textbf{--}} $s=2$}};
	%		\node at (0.7,0.9) {{\footnotesize {\color{green} \textbf{--}} $s=4$}};
	\begin{axis}[
	width=4.4cm,height=3.5cm,scale=1,
	x label style={at={(axis description cs:0.5,0.05)},anchor=north},
	y label style={at={(axis description cs:0.15,.5)},anchor=south},
	        xlabel={Iterations},
			ylabel={$\lambda_k = 1/k$},
	ticklabel style = {font=\tiny},
	%		ymode = log,
	xmode = log,
	ymin = 0, ymax=1,
	%    each nth point=500, filter discard warning=false, unbounded coords=discard ,
	xmin = 1e0, xmax=1e7, 
	every axis plot/.append style={line width=1pt}],
	legend pos=south west;
	\addplot [blue]         table [x index=20,y index=0]{decay_1k.dat};
	\addplot [green]        table [x index=20,y index=2]{decay_1k.dat};
	\addplot [black]    	 table [x index=20,y index=4]{decay_1k.dat};
	\addplot [yellow]        table [x index=20,y index=6]{decay_1k.dat};
	\addplot [red]         table [x index=20,y index=8]{decay_1k.dat};
	\end{axis}
	\end{tikzpicture}
	\begin{tikzpicture}
	%        \node at (1.5,3.3) { $F(x_k) - F^*$};
	\begin{axis}[
	width=4.4cm,height=3.5cm,scale=1,
	x label style={at={(axis description cs:0.5,0.05)},anchor=north},
	y label style={at={(axis description cs:0.15,.5)},anchor=south},
	        xlabel={Iterations},
	ticklabel style = {font=\tiny},
	%		ymode = log,
	xmode = log,
	ymin = 0, ymax=1,
	%    each nth point=500, filter discard warning=false, unbounded coords=discard ,
	xmin = 1e0, xmax=1e7, 
	every axis plot/.append style={line width=1pt}],
	legend pos=south west;
	\addplot [blue]         table [x index=20,y index=10]{decay_1k.dat};
	\addplot [green]        table [x index=20,y index=12]{decay_1k.dat};
	\addplot [black]    	 table [x index=20,y index=14]{decay_1k.dat};
	\addplot [yellow]        table [x index=20,y index=16]{decay_1k.dat};
	\addplot [red]         table [x index=20,y index=18]{decay_1k.dat};
	\end{axis}
	\end{tikzpicture}
	\caption{The effect of the rate at which the weights decay.}
	\label{fig:s2}
\end{figure}
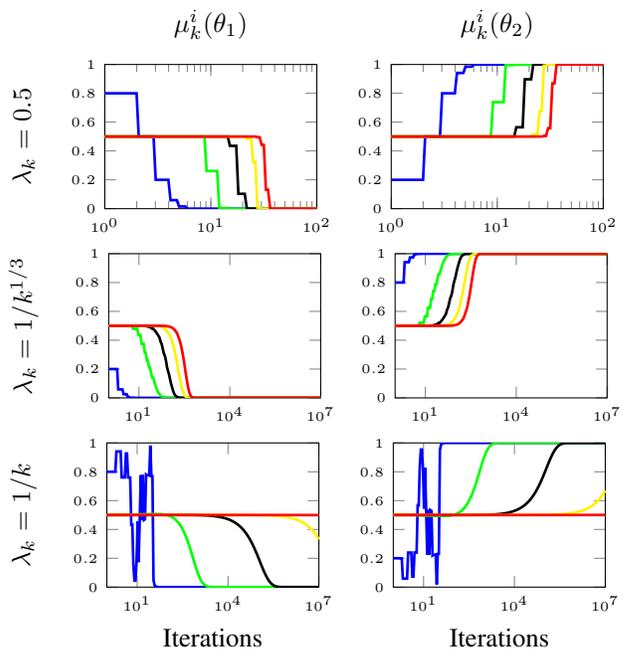

\section{CONCLUSIONS AND FUTURE WORK}\label{sec:conclusions}

We studied the problem of non-Bayesian learning for agents with increasing self-confidence. In our main result, we explicitly characterize the fastest rate at which an agent increases it self-confidence, or decreases the weights of its neighbors, and still guarantees that social learning occurs. Moreover, we do so for the learning problem where agents can have conflicting hypotheses. Finally, we show a converse of our main result, that states that if the rate at which the weight decreases is faster than our rate bound, there exist graphs for which no social learning occurs.

Two main questions remain open and require further work. First, what is the non-asymptotic convergence rate of beliefs generated by the log-linear update rule~\eqref{eq:update} when social learning occurs? How does this convergence rate depend on the convergence rate of the sequence $\{\lambda_k\}$? Second, if social learning does not occur, is it possible to estimate the distance between the infinite product of the matrices $\{T_k\}$ and a rank one matrix?

\bibliographystyle{ieeetr}
\bibliography{all_refs}

\end{document}